\documentclass[12pt]{amsart}
\usepackage{cases}
\usepackage{amsthm}
\usepackage{amsmath}
\usepackage{amscd}
\usepackage{graphicx}
\usepackage[mathscr]{eucal}
\usepackage[colorlinks,linkcolor=blue,citecolor=blue, pdfstartview=FitH]{hyperref}

\input xy
\xyoption{all} \numberwithin{equation}{section}
\begin{document}

\title[On the asymptotic expansions II]
{On the asymptotic expansions of various quantum invariants II: the colored Jones polynomial of twist knots at the root of unity $e^{\frac{2\pi\sqrt{-1}}{N+\frac{1}{M}}}$ and $e^{\frac{2\pi\sqrt{-1}}{N}}$}
\author[Qingtao Chen and Shengmao Zhu]{Qingtao Chen and
Shengmao Zhu}

\address{Division of Science \\
New York University Abu Dhabi \\
Abu Dhabi \\
United Arab Emirates}
\email{chenqtao@nyu.edu}, \email{chenqtao@hotmail.com}

\address{Department of Mathematics \\
Zhejiang Normal University  \\
Jinhua Zhejiang,  321004, China }
\email{szhu@zju.edu.cn}

\begin{abstract} 
This is the second article in a series devoted to the study of the asymptotic expansions of various quantum invariants related to the twist knots. In this article, following the method and results in \cite{CZ23-1}, we present an asymptotic expansion formula for the colored Jones polynomial of twist knot $\mathcal{K}_p$ with $p\geq 6$ at the root of unity $e^{\frac{2\pi\sqrt{-1}}{N+\frac{1}{M}}}$ with $M\geq 2$.
Furthermore, by taking the limit $M\rightarrow +\infty$, we obtain an asymptotic expansion formula for the colored Jones polynomial of twist knots $\mathcal{K}_p$ with $p\geq 6$ at the root of unity $e^{\frac{2\pi\sqrt{-1}}{N}}$. 
\end{abstract}

\maketitle

\theoremstyle{plain} \newtheorem{thm}{Theorem}[section] \newtheorem{theorem}[%
thm]{Theorem} \newtheorem{lemma}[thm]{Lemma} \newtheorem{corollary}[thm]{%
Corollary} \newtheorem{proposition}[thm]{Proposition} \newtheorem{conjecture}%
[thm]{Conjecture} \theoremstyle{definition}
\newtheorem{remark}[thm]{Remark}
\newtheorem{remarks}[thm]{Remarks} \newtheorem{definition}[thm]{Definition}
\newtheorem{example}[thm]{Example}





\tableofcontents
\newpage

\section{Introduction}
In the first paper of this series \cite{CZ23-1}, we have gotten an asymptotic expansion formula of the colored Jones polynomial for twist knot at the root of unity $e^{\frac{2\pi\sqrt{-1}}{N+\frac{1}{2}}}$. The motivation of that work \cite{CZ23-1} is to study a version of volume conjecture for colored Jones polynomial  proposed in \cite{DKY18} which states that  for a hyperbolic link $\mathcal{L}$ in $S^3$, we have 
   \begin{align}
       \lim_{N\rightarrow \infty}\frac{2\pi}{N} \log|J_{N}(\mathcal{L}; e^{\frac{2\pi\sqrt{-1}}{N+\frac{1}{2}}})|=vol(S^3\setminus \mathcal{L}).
   \end{align} 

In the present paper, we first consider the colored Jones polynomial at more general root of unity $e^{\frac{2\pi\sqrt{-1}}{N+\frac{1}{M}}}$ with $M\geq 2$.  We obtain an similar asymptotic expansion formula for the colored Jones polynomial for the twist knot $\mathcal{K}_p$ with $p\geq 6$ similar to the case of $M=2$ in \cite{CZ23-1}. The advantage of using the root of unity $e^{\frac{2\pi\sqrt{-1}}{N+\frac{1}{M}}}$ is that now we have an additional parameter $M$, so we can take the limit $\lim_{M \rightarrow +\infty}e^{\frac{2\pi\sqrt{-1}}{N+\frac{1}{M}}}=e^{\frac{2\pi\sqrt{-1}}{N}}$. Hence it provide a new way to study the asymptotic expansion of the colored Jones polynomial at the root of unity $e^{\frac{2\pi\sqrt{-1}}{N}}$.  It is well-known that the original Volume Conjecture due to Kashaev-Murakami-Murakami \cite{Kash97,MuMu01} was proposed at the root of unity  $e^{\frac{2\pi\sqrt{-1}}{N}}$. 

 In the second part of this paper, we prove that it is reasonable to take the limit, so we obtain an asymptotic expansion formula for colored Jones polynomial of twist knot $\mathcal{K}_p$ with $p\geq 6$ at the root of unity $e^{\frac{2\pi\sqrt{-1}}{N}}$. This work makes a connection between the asymptotic expansion of the colored Jones polynomial at different roots of unity. 
 
 The main results of this paper are as follows. 
Let $V(p,t,s)$ be the potential function of the colored Jones polynomial for the twist knot $\mathcal{K}_p$ given by formula (\ref{formula-Vpts}). 
By Proposition \ref{prop-critical}, there exists a unique critical point $(t_0,s_0)$ of $V(p,t,s)$. Let $x_0=e^{2\pi\sqrt{-1}t_0}$ and $y_0=e^{2\pi\sqrt{-1}s_0}$, we put
\begin{align}
  \zeta(p)&=V(p,t_0,s_0)\\\nonumber
  &=\pi \sqrt{-1}\left((2p+1)s_0^2-(2p+3)s_0-2t_0\right)\\\nonumber
    &+\frac{1}{2\pi\sqrt{-1}}\left(\text{Li}_2(x_0y_0)+\text{Li}_2(x_0/y_0)-3\text{Li}_2(x_0)+\frac{\pi^2}{6}\right)  
\end{align}
and 
\begin{align} \label{formula-omega0}
    \omega(p)&=\frac{\sin (2\pi s_0)e^{2\pi\sqrt{-1}t_0}}{(1-e^{2\pi\sqrt{-1}t_0})^{\frac{3}{2}}\sqrt{\det Hess(V)(t_0,s_0)}}\\\nonumber
    &=\frac{(y_0-y_0^{-1})x_0}{-4\pi (1-x_0)^\frac{3}{2}\sqrt{H(p,x_0,y_0)}}
\end{align}
with
\begin{align}
    H(p,x_0,y_0)&=\left(\frac{-3(2p+1)}{\frac{1}{x_0}-1}+\frac{2p+1}{\frac{1}{x_0y_0}-1}+\frac{2p+1}{\frac{1}{x_0/y_0}-1}-\frac{3}{(\frac{1}{x_0}-1)(\frac{1}{x_0y_0}-1)}\right.\\\nonumber
    &\left.-\frac{3}{(\frac{1}{x_0}-1)(\frac{1}{x_0/y_0}-1)}+\frac{4}{(\frac{1}{x_0y_0}-1)(\frac{1}{x_0/y_0}-1)}\right).
\end{align}

\begin{theorem}  \label{theorem-main}
   For $p\geq 6$ and $M\geq 2$, the asymptotic expansion of the colored Jones polynomial  of the twist knot $\mathcal{K}_p$ at the root of unity $e^{\frac{2\pi \sqrt{-1}}{N+\frac{1}{M}}}$ is given by the following form
    \begin{align}
J_N(\mathcal{K}_p;e^{\frac{2\pi \sqrt{-1}}{N+\frac{1}{M}}})&=(-1)^{p}\frac{4\pi e^{\pi\sqrt{-1}(\frac{1}{4}+\frac{2}{M})}(N+\frac{1}{M})^{\frac{1}{2}}\sin\frac{\pi}{M}}{\sin\frac{\frac{\pi}{M}}{N+\frac{1}{M}}}\omega(p)e^{(N+\frac{1}{M})\zeta(p)}\\\nonumber
       &\cdot\left(1+\sum_{i=1}^d\kappa_i(p,\frac{1}{M})\left(\frac{2\pi\sqrt{-1}}{N+\frac{1}{M}}\right)^i+O\left(\frac{1}{(N+\frac{1}{M})^{d+1}}\right)\right),
    \end{align}
    for $d\geq 1$, where $\omega(p)$ and $\kappa_i(p,\frac{1}{M})$ are constants determined by $\mathcal{K}_p$, and $\omega(p)$ is given by formula (\ref{formula-omega0}).
\end{theorem}

Theorem \ref{theorem-main} holds for any $M\geq 2$, note that we have the following 
\begin{align}
    \lim_{M\rightarrow \infty}J_N(\mathcal{K}_p;e^{\frac{2\pi \sqrt{-1}}{N+\frac{1}{M}}})=J_N(\mathcal{K}_p;e^{\frac{2\pi \sqrt{-1}}{N}}),
\end{align}
since the colored Jones polynomial $J_N(\mathcal{K}_p;q)$ is a polynomial of $q$. Then we prove that under the limit $M\rightarrow +\infty$, we have 
\begin{theorem} \label{theorem-main2}
For $p\geq 6$, the asymptotic expansion of the colored Jones polynomial of the twist knot $\mathcal{K}_p$ at the root of unity $e^{\frac{2\pi \sqrt{-1}}{N}}$ is given by  the following form
\begin{align}
  J_{N}(\mathcal{K}_p;e^{\frac{2\pi \sqrt{-1}}{N}})&=(-1)^p4\pi e^{\frac{\pi\sqrt{-1}}{4}}N^{\frac{3}{2}}\omega(p)e^{N\zeta(p)}\\\nonumber
       &\cdot\left(1+\sum_{i=1}^d\kappa_i(p)\left(\frac{2\pi\sqrt{-1}}{N}\right)^i+O\left(\frac{1}{N^{d+1}}\right)\right),
    \end{align}
    for $d\geq 1$, where $\omega(p)$ and $\kappa_i(p)$ are constants determined by $\mathcal{K}_p$,  and $\omega(p)$ is given by formula (\ref{formula-omega0}).
\end{theorem}

In \cite{CZ23-1},  we have proved the following
\begin{lemma} [\cite{CZ23-1}, Lemma 5.4]
\begin{align}
     2\pi \zeta(p)=vol(S^3\setminus \mathcal{K}_p)+\sqrt{-1}cs(S^3\setminus \mathcal{K}_p) \mod \pi^2\sqrt{-1}\mathbb{Z}.
\end{align}
where $vol(S^3\setminus \mathcal{K}_p)$ denotes the hyperbolic volume of the complement of $\mathcal{K}_p$ in $S^3$ and $cs(S^3\setminus \mathcal{K}_p)$ denotes the Chern-Simons invariant. 
\end{lemma}
Then, Theorem \ref{theorem-main2} implies that
\begin{corollary}
    For $p\geq 6$, we have
    \begin{align}
        \lim_{N\rightarrow \infty}\frac{2\pi}{N}\log J_N(\mathcal{K}_p;e^{\frac{2\pi \sqrt{-1}}{N}})=vol(S^3\setminus \mathcal{K}_p)+\sqrt{-1}cs(S^3\setminus \mathcal{K}_p) \mod \pi^2\sqrt{-1}\mathbb{Z}.
    \end{align}
\end{corollary}
Hence we prove Kashaev-Murakami-Murakami Volume Conjecture for twist knot $\mathcal{K}_p$ with  $p\geq 6$.

\begin{remark}
    We need the condition $p\geq 6$ in the above two theorems since we use the same method from previous work \cite{CZ23-1}. We remark that this method can also work for the cases of $p\leq -1$ with some exceptions.     
\end{remark}

The rest of this article is organized as follows. In Section \ref{Section-Prelim}, we fix the notations and review the related materials that will be used in this paper. In Section \ref{Section-potential}, we compute the potential function for the colored Jones polynomials of the twist knot $\mathcal{K}_p$ at the root of unity $e^{\frac{2\pi\sqrt{-1}}{N+\frac{1}{M}}}$. In Section  \ref{Section-Poisson}, we prove Proposition \ref{prop-fouriercoeff}  which expresses the colored Jones polynomial of the twist knot 
$J_N(\mathcal{K}_p;e^{\frac{2\pi\sqrt{-1}}{N+\frac{1}{M}}})$ as a summation of Fourier coefficients by Poisson summation formula.
Section \ref{Section-Asym1} is devoted to the study of the asymptotic expansion of the colored Jones polynomial $J_N(\mathcal{K}_p;e^{\frac{2\pi\sqrt{-1}}{N+\frac{1}{M}}})$ of the twist knot $\mathcal{K}_p$ 
by using the results directly from \cite{CZ23-1}. In Section \ref{Section-Asym2}, we obtain the asymptotic expansion of the colored Jones polynomial $J_N(\mathcal{K}_p;e^{\frac{2\pi\sqrt{-1}}{N}})$ of the twist knot $\mathcal{K}_p$
 by taking the limit $M\rightarrow +\infty$.

\textbf{Acknowledgements.} 

The first author would like to thank Nicolai Reshetikhin, Kefeng Liu and Weiping Zhang for bringing him to this area and a lot of discussions during his career, thank Francis Bonahon,   Giovanni Felder and Shing-Tung Yau for their continuous encouragement, support and discussions, and thank Jun Murakami and Tomotada Ohtsuki for their helpful discussions and support. He also want to thank Jørgen Ellegaard Andersen, Sergei Gukov, Thang Le, Gregor Masbaum,  Rinat Kashaev, Vladimir Turaev and Hiraku Nakajima for their support, discussions and interests, and thank Yunlong Yao who built him solid analysis foundation twenty years ago.

The second author would like to thank Kefeng Liu and Hao Xu for bringing him to this area when he was a graduate student at CMS of Zhejiang University, and for their constant encouragement and helpful discussions since then.

\section{Preliminaries} \label{Section-Prelim}
\subsection{Colored Jones polynomials of the twist knot $\mathcal{K}_p$}

We consider the twist knot $\mathcal{K}_p$ illustrated in Figure 1,
\begin{figure}[!htb] \label{figure1} 
\begin{align} 
\raisebox{-15pt}{
\includegraphics[width=120 pt]{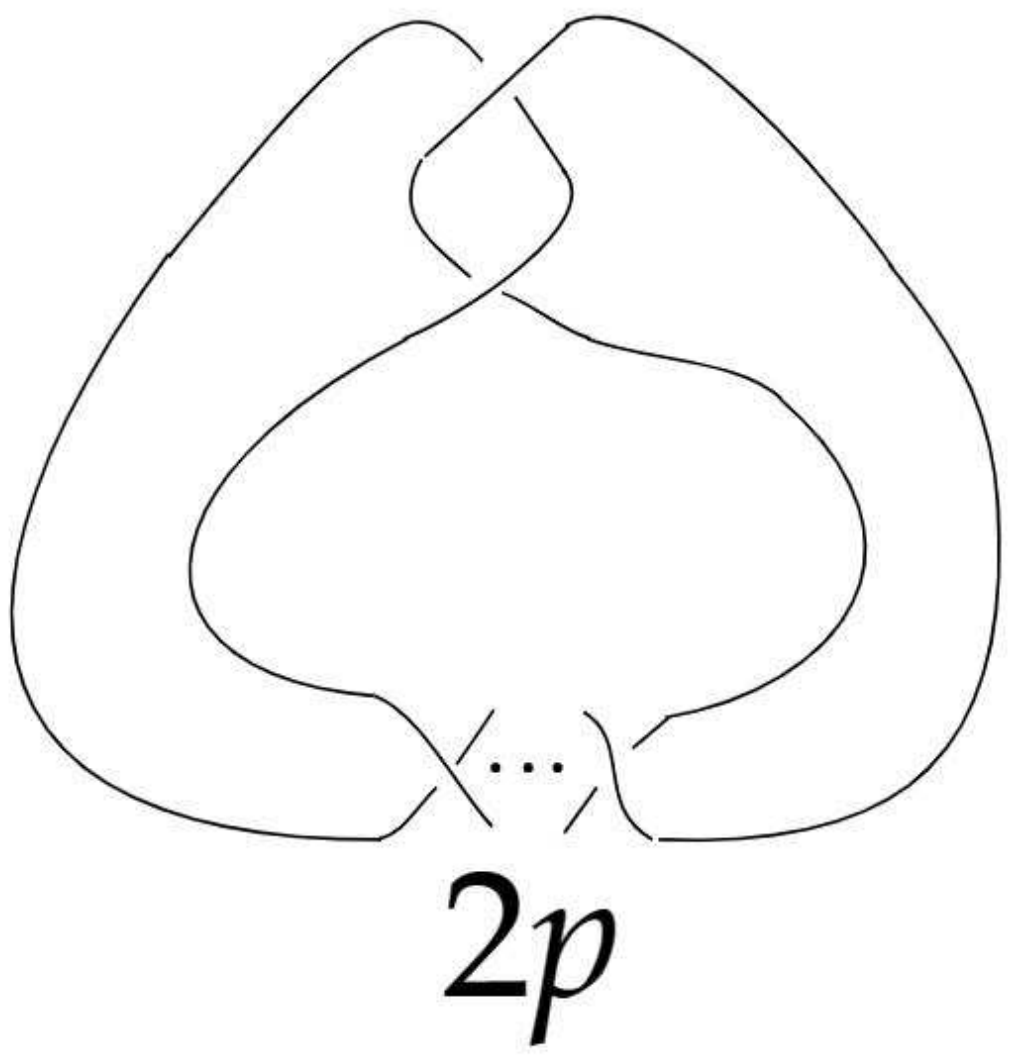}}.
\end{align}
\caption{Twist knot $\mathcal{K}_p$}
\end{figure}
where the index $2p$ represents $2p$ crossings (half-twists). For example, $\mathcal{K}_{-1}=4_1$, $\mathcal{K}_1=3_1$, $\mathcal{K}_2=5_2$.

We will use the following formula for 
the {\em normalized $N$-colored Jones polynomial} of the twist knot $\mathcal{K}_p$ given by K. Habiro and G. Masbaum in \cite{Hab08,Mas03}
\begin{align} \label{formula-coloredJonestwist}
J_N(\mathcal{K}_{p};q)=\sum_{k=0}^{N-1}\sum_{l=0}^{k}(-1)^lq^{\frac{k(k+3)}{4}+pl(l+1)}\frac{\{k\}!\{2l+1\}}{\{k+l+1\}!\{k-l\}!}\prod_{i=1}^k(\{N+i\}
\{N-i\}),
\end{align}
where 
\begin{align}
    \{n\}=q^{\frac{n}{2}}-q^{-\frac{n}{2}}, \ \ \text{for a positive integer $n$}. 
\end{align}

\subsection{Dilogarithm and Lobachevsky functions}
Let $\log: \mathbb{C}\setminus (-\infty,0]\rightarrow \mathbb{C}$ be the standard logarithm function defined by 
\begin{align}
    \log z=\log |z|+\sqrt{-1}\arg z
\end{align}
with $-\pi <\arg z<\pi$. 

The dilogarithm function $\text{Li}_2: \mathbb{C}\setminus (1,\infty)\rightarrow \mathbb{C}$ is defined by 
\begin{align}
    \text{Li}_2(z)=-\int_0^{z}\frac{\log(1-x)}{x}dx
\end{align}
where the integral is along any path in $\mathbb{C}\setminus (1,\infty)$ connecting $0$ and $z$, which is holomorphic in $\mathbb{C}\setminus [1,\infty)$ and continuous in $\mathbb{C}\setminus (1,\infty)$. 

The dilogarithm function satisfies the following properties 
\begin{align}
    \text{Li}_2\left(\frac{1}{z}\right)=-\text{Li}_2(z)-\frac{\pi^2}{6}-\frac{1}{2}(\log(-z) )^2.
\end{align}
In the unit disk $\{z\in \mathbb{C}| |z|<1\}$,  $\text{Li}_2(z)=\sum_{n=1}^{\infty}\frac{z^n}{n^2}$, and on the unit circle 
\begin{align}
 \{z=e^{2\pi \sqrt{-1}t}|0 \leq t\leq 1\},    
\end{align}
we have
\begin{align}
    \text{Li}_2(e^{2\pi\sqrt{-1} t})=\frac{\pi^2}{6}+\pi^2t(t-1)+2\pi \sqrt{-1}\Lambda(t)
\end{align}
where 
\begin{align} \label{formula-Lambda(t)}
\Lambda(t)=\text{Re}\left(\frac{\text{Li}_2(e^{2\pi \sqrt{-1}t})}{2\pi \sqrt{-1}}\right)=-\int_{0}^{t}\log|2 \sin \pi t|d t 
\end{align}
for $t\in \mathbb{R}$. The function $\Lambda(t)$ is an odd function which has period $1$ and satisfies 
$
\Lambda(1)=\Lambda(\frac{1}{2})=0.
$

\subsection{Quantum dilogrithm functions}
Given two positive integers $N$ and $M$ with $M\geq 2$, we set
$$\xi_{N,\frac{1}{M}}=e^{\frac{2\pi\sqrt{-1}}{N+\frac{1}{M}}}.$$

  We introduce the holomorphic function $\varphi_{N,\frac{1}{M}}(t)$ for $\{t\in
\mathbb{C}| 0< \text{Re} t < 1\}$, by the following integral
\begin{align}
\varphi_{N,\frac{1}{M}}(t)=\int_{-\infty}^{+\infty}\frac{e^{(2t-1)x}dx}{4x \sinh x
\sinh\frac{x}{N+\frac{1}{M}}}.
\end{align}
Noting that this integrand has poles at $n\pi \sqrt{-1} (n\in
\mathbb{Z})$, where, to avoid the poles at $0$, we choose the
following contour of the integral
\begin{align}
\gamma=(-\infty,-1]\cup \{z\in \mathbb{C}||z|=1, \text{Im} z\geq 0\}
\cup [1,\infty).
\end{align}

\begin{lemma} \label{lemma-varphixi}
The function $\varphi_{N,\frac{1}{M}}(t)$ satisfies
\begin{align}
    (\xi_{N,\frac{1}{M}})_n=\exp \left(\varphi_{N,\frac{1}{M}}\left(\frac{1}{2\left(N+\frac{1}{M}\right)}\right)-\varphi_{N,\frac{1}{M}}\left(\frac{2n+1}{2\left(N+\frac{1}{M}\right)}\right)\right)
    \end{align}
for  $0\leq n\leq N$, and
    \begin{align}
        (\xi_{N,\frac{1}{M}})_n
        &=\exp \left(\varphi_{N,\frac{1}{M}}\left(\frac{1}{2\left(N+\frac{1}{M}\right)}\right)-\varphi_{N,\frac{1}{M}}\left(\frac{2n+1}{2\left(N+\frac{1}{M}\right)}-1\right)\right.\\\nonumber
        &\left.+\log \left(1-e^{-\frac{2\pi\sqrt{-1}}{M}}\right)\right) \ \text{ for $ N< n\leq 2N$}.
    \end{align}

\end{lemma}

\begin{lemma} \label{lemma-varphixi2}
    We have the following identities:
\begin{align}
    &\varphi_{N,\frac{1}{M}}(t)+\varphi_{N,\frac{1}{M}}(1-t)\\\nonumber
    &=2\pi \sqrt{-1}\left(-\frac{(N+\frac{1}{M})}{2}(t^2-t+\frac{1}{6})+\frac{1}{24(N+\frac{1}{M})}\right),\\ 
    &\varphi_{N,\frac{1}{M}}\left(\frac{1}{2\left(N+\frac{1}{M}\right)}\right)\\\nonumber
    &=\frac{(N+\frac{1}{M})}{2\pi\sqrt{-1}}\frac{\pi^2}{6}+\frac{1}{2}\log \left(N+\frac{1}{M}\right)+\frac{\pi \sqrt{-1}}{4}-\frac{\pi \sqrt{-1}}{12(N+\frac{1}{M})},\\
    &\varphi_{N,\frac{1}{M}}\left(1-\frac{1}{2\left(N+\frac{1}{M}\right)}\right)\\\nonumber
    &=\frac{\left(N+\frac{1}{M}\right)}{2\pi\sqrt{-1}}\frac{\pi^2}{6}-\frac{1}{2}\log \left(N+\frac{1}{M}\right)+\frac{\pi \sqrt{-1}}{4}-\frac{\pi \sqrt{-1}}{12\left(N+\frac{1}{M}\right)}.
\end{align}
\end{lemma}
The function $\varphi_{N,\frac{1}{M}}(t)$ is closely related to the dilogarithm function as follows.
\begin{lemma} \label{lemma-varphixi3}
    (1)For every $t$ with $0<Re t<1$, 
    \begin{align}
        \varphi_{N,\frac{1}{M}}(t)=\frac{(N+\frac{1}{M})}{2\pi \sqrt{-1}}\text{Li}_2(e^{2\pi\sqrt{-1}t})
 -\frac{\pi \sqrt{-1}e^{2\pi\sqrt{-1}t}}{12(1-e^{2\pi\sqrt{-1}t})}\frac{1}{N+\frac{1}{M}}+O\left(\frac{1}{(N+\frac{1}{M})^3}\right).
    \end{align}
    (2) For every $t$ with $0<Re t<1$, 
    \begin{align}
        \varphi_{N,\frac{1}{M}}'(t)=-\left(N+\frac{1}{M}\right)\log(1-e^{2\pi\sqrt{-1}t})+O\left(\frac{1}{\left(N+\frac{1}{M}\right)}\right)
    \end{align}
    (3) As $N\rightarrow \infty$, $\frac{1}{\left(N+\frac{1}{M}\right)}\varphi_{N,\frac{1}{M}}(t)$ uniformly converges to $\frac{1}{2\pi\sqrt{-1}}\text{Li}_2(e^{2\pi\sqrt{-1}t})$ and $\frac{1}{\left(N+\frac{1}{M}\right)}\varphi'_{N,M}(t)$ uniformly converges to $-\log(1-e^{2\pi\sqrt{-1}t})$ on any compact subset of $\{t\in \mathbb{C}|0<Re t<1\}$. 
\end{lemma}
See the literature, such as \cite{Oht16,CJ17,WongYang20-1} for the proof of Lemma \ref{lemma-varphixi}, \ref{lemma-varphixi2}, \ref{lemma-varphixi3}.

 \subsection{Saddle point method}
We need to use the following version of saddle point method as shown in \cite{Oht18}.
\begin{proposition}[\cite{Oht18}, Proposition 3.1] \label{proposition-saddlemethod}
   Let $A$ be a non-singular symmetric complex $2\times 2$ matrix, and let $\Psi(z_1,z_2)$ and $r(z_1,z_2)$ be holomorphic functions of the forms, 
   \begin{align}
    \Psi(z_1,z_2)&=\mathbf{z}^{T}A\mathbf{z}+r(z_1,z_2), \\\nonumber
    r(z_1,z_2)&=\sum_{i,j,k}b_{ijk}z_iz_jz_k+\sum_{i,j,k,l}c_{ijkl}z_iz_jz_kz_l+\cdots
   \end{align}
   defined in a neighborhood of $\mathbf{0}\in \mathbb{C}$. The restriction of the domain 
   \begin{align} \label{formula-domain0}
       \{(z_1,z_2)\in \mathbb{C}^2| \text{Re}\Psi(z_1,z_2)<0\}  
   \end{align}
   to a neighborhood of $\mathbf{0}\in \mathbb{C}^2$ is homotopy equivalent to $S^1$. Let $D$ be an oriented disk embeded in $\mathbb{C}^2$ such that $\partial D$ is included in the domain (\ref{formula-domain0}) whose inclusion is homotopic to a homotopy equivalence to the above $S^1$ in the domain (\ref{formula-domain0}). Then we have the following asymptotic expansion
\begin{align}
    \int_{D}e^{N\psi(z_1,z_2)}dz_1dz_2=\frac{\pi}{N\sqrt{\det(-A)}}\left(1+\sum_{i=1}^d\frac{\lambda_i}{N^i}+O(\frac{1}{N^{d+1}})\right),
\end{align}
   for any $d$, where we choose the sign of $\sqrt{\det{(-A)}}$ as explained in Proposition \cite{Oht16}, and $\lambda_i$'s are constants presented by using coefficients of the expansion $\Psi(z_1,z_2)$, such presentations are obtained by formally expanding the following formula, 
\begin{align}
    1+\sum_{i=1}^{\infty}\frac{\lambda_i}{N^i}=\exp\left(Nr\left(\frac{\partial }{\partial w_1},\frac{\partial }{\partial w_2}\right)\right)\exp\left(-\frac{1}{4N}(w_1,w_2)A^{-1}\binom{w_1}{w_2}\right)|_{w_1=w_2=0}.
\end{align}
\end{proposition}
See \cite{Oht16} for a proof of the Proposition \ref{proposition-saddlemethod},   

\begin{remark}[\cite{Oht18}, Remark 3.2] \label{remark-saddle}
    As mentioned in Remark 3.6 of \cite{Oht16}, we can extend Proposition \ref{proposition-saddlemethod} to the case where $\Psi(z_1,z_2)$ depends on $N$ in such a way that $\Psi(z_1,z_2)$ is of the form 
    \begin{align}        \Psi(z_1,z_2)=\Psi_0(z_1,z_2)+\Psi_1(z_1,z_2)\frac{1}{N}+R(z_1,z_2)\frac{1}{N^2}. 
    \end{align}
    where $\Psi_i(z_1,z_2)$'s are holomorphic functions independent of $N$, and we assume that $\Psi_0(z_1,z_2)$ satisfies the assumption of the Proposition and $|R(z_1,z_2)|$ is bounded by a constant which is independent of $N$.  
\end{remark}
\subsection{Conventions and Notations}
We list the following notations used in this paper and comparing to \cite{CZ23-1}. 

The roots of unity: 
\begin{itemize}
    \item[] $\xi_{N,\frac{1}{M}}=e^{\frac{2\pi\sqrt{-1}}{N+\frac{1}{M}}}$ ( so $\xi_{N,\frac{1}{2}}$ is equal to the notation $\xi_N$ used in \cite{CZ23-1});

    \item[] $\xi_{N,0}=e^{\frac{2\pi\sqrt{-1}}{N}}$.
\end{itemize}

The Faddeev functions:
\begin{itemize}
    \item[] $\varphi_{N,\frac{1}{M}}(t)\int_{-\infty}^{+\infty}\frac{e^{(2t-1)x}dx}{4x \sinh x
\sinh\frac{x}{N+\frac{1}{M}}}$ ( so $\varphi_{N,\frac{1}{2}}(t)$ is equal to the notation $\varphi_{N}(t)$ used in \cite{CZ23-1});

    \item[] $\varphi_{N,0}(t)=\int_{-\infty}^{+\infty}\frac{e^{(2t-1)x}dx}{4x \sinh x
\sinh\frac{x}{N}}$.
\end{itemize}

The Potential functions:
\begin{itemize}
    \item[] $V_{N\frac{1}{M}}(p,t,s)$ is given by formula (\ref{formula-potentialVNM})  ( so $V_{N,\frac{1}{2}}(p,t,s)$ is equal to the notation $V_{N}(p,t,s)$ used in \cite{CZ23-1});

    \item[] $V_{N,0}(p,t,s)$ is given by the formula (\ref{formula-VN0}). 
\end{itemize}

The Fourier coefficients:
\begin{itemize}
    \item[] $\hat{h}_{N,\frac{1}{M}}(m,n)$ is given by formula (\ref{formula-hathMN})  ( so $\hat{h}_{N,\frac{1}{2}}(m,n)$ is equal to the notation $\hat{h}_{N}(m,n)$ used in \cite{CZ23-1});

    \item[] $\tilde{h}_{N,\frac{1}{M}}(m,n)=(1-e^{\frac{2\pi\sqrt{-1}(n+1)}{M}})\hat{h}_{N,\frac{1}{M}}(m,n)$;

    \item[] 
    $\tilde{h}_{N,0}(m,n)$ is given by formula (\ref{formula-tildehN}). 
\end{itemize}

\section{Calculations of the potential function} \label{Section-potential}
This section is devoted to the calculations of the potential function for the colored 
Jones polynomial $J_{N}(\mathcal{K}_p;q)$ at the root of unity $\xi_{N,\frac{1}{M}}=e^{\frac{2\pi\sqrt{-1}}{N+\frac{1}{M}}}$. 

We introduce the following $q$-Pochhammer symbol 
\begin{align}
    (q)_n=\prod_{i=1}^{n}(1-q^i).
\end{align}
then we have
\begin{align}
\ \{n\}!=(-1)^nq^{\frac{-n(n+1)}{4}}(q)_n.
\end{align}

 From formula (\ref{formula-coloredJonestwist}),  we obtain
\begin{align}
J_{N}(\mathcal{K}_p;q)&=\sum_{k=0}^{N-1}\sum_{l=0}^k(-1)^{k+l}q^{pl(l+1)+\frac{l(l-1)}{2}-Nk+\frac{k(k+1)}{2}+k}\\\nonumber
&\cdot \frac{(1-q^{2l+1})}{(1-q^N)}\frac{(q)_k(q)_{N+k}}{(q)_{k+l+1}(q)_{k-l}(q)_{N-k-1}}. 
\end{align}
Hence, at the root of unity $\xi_{N,\frac{1}{M}}$,  we have 
\begin{align}
 J_{N}(\mathcal{K}_p;\xi_{N,\frac{1}{M}})&= \sum_{k=0}^{N-1}\sum_{l=0}^k\frac{(-1)^{k+l+1}\sin \frac{\pi(2l+1)}{N+\frac{1}{M}}}{\sin \frac{\frac{\pi}{M} }{N+\frac{1}{M}}}\\\nonumber
 &\cdot \xi_{N,\frac{1}{M}}^{(p+\frac{1}{2})l^2+(p+\frac{1}{2})l+\frac{k^2}{2}+2k+\frac{3}{4}}\frac{(\xi_{N,\frac{1}{M}})_k(\xi_{N,\frac{1}{M}})_{N+k}}{(\xi_{N,\frac{1}{M}})_{k+l+1}(\xi_{N,\frac{1}{M}})_{k-l}(\xi_{N,\frac{1}{M}})_{N-k-1}}. 
\end{align}

By using Lemma \ref{lemma-varphixi},  we obtain   

\begin{align}    &\frac{(\xi_{N,\frac{1}{M}})_k(\xi_{N,\frac{1}{M}})_{N+k}}{(\xi_{N,\frac{1}{M}})_{k+l+1}(\xi_{N,\frac{1}{M}})_{k-l}(\xi_{N,\frac{1}{M}})_{N-k-1}}
    \\\nonumber
&=\exp\left(\varphi_{N,\frac{1}{M}}\left(\frac{2(k+l+1)+1}{2(N+\frac{1}{M})}-1\right)+\varphi_{N,\frac{1}{M}}\left(\frac{2(k-l)+1}{2(N+\frac{1}{M})}\right)\right.\\\nonumber
    &\left.+\varphi_{N,\frac{1}{M}}\left(1-\frac{2k+1+\frac{2}{M}}{2(N+\frac{1}{M})}\right)-\varphi_{N,\frac{1}{M}}\left(\frac{2k+1}{2(N+\frac{1}{M})}\right)\right.\\\nonumber&\left.-\varphi_{N,\frac{1}{M}}\left(\frac{2k+1-\frac{2}{M}}{2(N+\frac{1}{M})}\right)-\varphi_{N,\frac{1}{M}}\left(\frac{1}{2(N+\frac{1}{M})}\right)\right).
\end{align}
 for $N<k+l+1\leq 2N$. Similarly, one can obtain the expression for the case $0<k+l+1\leq N$, but which will not be used in the rest of this paper. 

By using Lemma \ref{lemma-varphixi2}, we  obtain 
\begin{align}
    &(-1)^{l-k-1}\xi_{N,\frac{1}{M}}^{(2p+1)(l^2+l)+k^2+4k+\frac{3}{2}}\frac{(\xi_{N,\frac{1}{M}})_k(\xi_{N,\frac{1}{M}})_{N+k}}{(\xi_{N,\frac{1}{M}})_{k+l+1}(\xi_{N,\frac{1}{M}})_{k-l}(\xi_{N,\frac{1}{M}})_{N-k-1}}\\\nonumber
    &=\exp (N+\frac{1}{M})\left(\frac{\pi \sqrt{-1}(2(2p+1)l^2+2(2p+1)l+(6-\frac{4}{M})k+3-2(\frac{1}{2}+\frac{1}{M})^2+\frac{1}{3})}{2(N+\frac{1}{M})^2}\right.\\\nonumber
    &\left.\frac{\pi\sqrt{-1}(l-\frac{3}{4}+\frac{1}{M})}{N+\frac{1}{M}}-\frac{\log(N+\frac{1}{M})}{2(N+\frac{1}{M})}-\frac{\pi\sqrt{-1}}{12}+\frac{1}{N+\frac{1}{M}}\varphi_{N,\frac{1}{M}}\left(\frac{k+l+\frac{3}{2}}{N+\frac{1}{M}}-1\right)\right.\\\nonumber
    &\left.+\frac{1}{N+\frac{1}{M}}\varphi_{N,\frac{1}{M}}\left(\frac{k-l+\frac{1}{2}}{N+\frac{1}{M}}\right)-\frac{1}{N+\frac{1}{M}}\varphi_{N,\frac{1}{M}}\left(\frac{k+\frac{1}{2}-\frac{1}{M}}{N+\frac{1}{M}}\right)\right.\\\nonumber&\left.-\frac{1}{N+\frac{1}{M}}\varphi_{N,\frac{1}{M}}\left(\frac{k+\frac{1}{2}}{N+\frac{1}{M}}\right)-\frac{1}{N+\frac{1}{M}}\varphi_{N,\frac{1}{M}}\left(\frac{k+\frac{1}{2}+\frac{1}{M}}{N+\frac{1}{M}}\right) \right)
\end{align}
for $N<k+l+1\leq 2N$. 

Now we set
\begin{align}
    t=\frac{k+\frac{1}{2}}{N+\frac{1}{M}}, s=\frac{l+\frac{1}{2}}{N+\frac{1}{M}},
\end{align}
and define the function $\tilde{V}_{N,\frac{1}{M}}(p,t,s)$ as follows.

For $0<t<1$, $0<t-s<1$ and $1<t+s<2$, we let 
\begin{align*}  
&\tilde{V}_{N,\frac{1}{M}}(p,t,s)\\\nonumber
&=\pi\sqrt{-1}((2p+1)s^2+s+(\frac{2}{N+\frac{1}{M}}-2)t-\frac{5-\frac{4}{M}}{4(N+\frac{1}{M})}-\frac{6p+4+\frac{12}{M^2}}{12(N+\frac{1}{M})^2}-\frac{1}{12})\\\nonumber
&-\frac{\log(N+\frac{1}{M})}{2(N+\frac{1}{M})}+\frac{1}{N+\frac{1}{M}}\left(\varphi_{N,\frac{1}{M}}(t-s+\frac{1}{2(N+\frac{1}{M})})+\varphi_{N,\frac{1}{M}}(t+s+\frac{1}{2(N+\frac{1}{M})}-1)\right.\\\nonumber&\left.-\varphi_{N,\frac{1}{M}}(t-\frac{1}{M(N+\frac{1}{M})})-\varphi_{N,\frac{1}{M}}(t)-\varphi_{N,\frac{1}{M}}(t+\frac{1}{M(N+\frac{1}{M})})\right)
\end{align*}

Based on the above calculations, we obtain
\begin{align} \label{formula-coloredJonesPotential1}
    &J_{N}(\mathcal{K}_p;\xi_{N,\frac{1}{M}})\\\nonumber
    &=\sum_{k=0}^{N-1}\sum_{l=0}^k\frac{\sin \frac{\pi (2l+1)}{N+\frac{1}{M}}}{\sin\frac{\frac{\pi}{M}}{N+\frac{1}{M}}}e^{(N+\frac{1}{M})\tilde{V}_{N,\frac{1}{M}}\left(\frac{k+\frac{1}{2}}{N+\frac{1}{M}},\frac{l+\frac{1}{2}}{N+\frac{1}{M}}\right)}\\\nonumber
    &=\sum_{k=0}^{N-1}\sum_{l=0}^k\frac{\sin \frac{\pi (2l+1)}{N+\frac{1}{M}}}{\sin\frac{\frac{\pi}{M}}{N+\frac{1}{M}}}e^{(N+\frac{1}{M})\left(\tilde{V}_{N,\frac{1}{M}}\left(\frac{k+\frac{1}{2}}{N+\frac{1}{M}},\frac{l+\frac{1}{2}}{N+\frac{1}{M}}\right)-2\pi\sqrt{-1}\frac{k}{N+\frac{1}{M}}-2(p+2)\pi\sqrt{-1}\frac{l}{N+\frac{1}{M}}\right)}. 
\end{align}
For convenience, we introduce the function $V_{N,\frac{1}{M}}(p,t,s)$ which is determined by the following formula  
\begin{align}
    &\tilde{V}_{N,\frac{1}{M}}(p,t,s)-2\pi\sqrt{-1}(t-\frac{\frac{1}{2}}{N+\frac{1}{M}})-2(p+2)\pi\sqrt{-1}(s-\frac{\frac{1}{2}}{N+\frac{1}{M}})\\\nonumber
    &=V_{N,\frac{1}{M}}(p,t,s)+\pi\sqrt{-1}\frac{4p+7+\frac{4}{M}}{4(N+\frac{1}{M})}-\frac{1}{2(N+\frac{1}{M})}\log\left(N+\frac{1}{M}\right). 
\end{align}

Note that the functions $\tilde{V}_{N,\frac{1}{M}}(p,t,s)$ and $V_{N,\frac{1}{M}}(p,t,s)$ are defined on the region 
  \begin{align}
    D=\{(t,s)\in \mathbb{R}^2| 0<t<1, 0<s<1,  0< t-s<1\}.
\end{align}

From formula (\ref{formula-coloredJonesPotential1}),  we finally obtain
\begin{proposition} 
The normalized $N$-th colored Jones polynomial of the twist $\mathcal{K}_p$ at the root of unit $\xi_{N,\frac{1}{M}}$  can be computed as
\begin{align} \label{formula-coloredJonespotential2}     J_N(\mathcal{K}_{p};\xi_{N,\frac{1}{M}})=\sum_{k=0}^{N-1}\sum_{l=0}^k g_{N,\frac{1}{M}}(k,l)
\end{align}
with
\begin{align}
    g_{N,\frac{1}{M}}(k,l)=(-1)^pe^{\pi\sqrt{-1}(\frac{1}{M}-\frac{1}{4})}\frac{1}{\sqrt{(N+\frac{1}{M})}}\frac{\sin \frac{\pi (2l+1)}{N+\frac{1}{M}}}{\sin\frac{\frac{\pi}{M}}{N+\frac{1}{M}}}e^{(N+\frac{1}{M})V_{N,\frac{1}{M}}\left(p,\frac{k+\frac{1}{2}}{N+\frac{1}{M}},\frac{l+\frac{1}{2}}{N+\frac{1}{M}}\right)},
\end{align}
where the function $V_{N,\frac{1}{M}}(p,t,s)$ is  given by  
\begin{align} \label{formula-potentialVNM}
    &V_{N,\frac{1}{M}}(p,t,s)\\\nonumber
    &=\pi \sqrt{-1}\left((2p+1)s^2-(2p+3)s+\left(\frac{2}{N+\frac{1}{M}}-2\right)t-\frac{6p+4+\frac{12}{M^2}}{12(N+\frac{1}{M})^2}\right)\\\nonumber
    &+\frac{1}{N+\frac{1}{M}}\varphi_{N,\frac{1}{M}}\left(t+s+\frac{\frac{1}{2}}{N+\frac{1}{M}}-1\right)+\frac{1}{N+\frac{1}{M}}\varphi_{N,\frac{1}{M}}\left(t-s+\frac{\frac{1}{2}}{N+\frac{1}{M}}\right)\\\nonumber
    &-\frac{1}{N+\frac{1}{M}}\varphi_{N,\frac{1}{M}}\left(t\right)-\frac{1}{N+\frac{1}{M}}\varphi_{N,\frac{1}{M}}\left(t-\frac{\frac{1}{M}}{N+\frac{1}{M}}\right)\\\nonumber
    &-\frac{1}{N+\frac{1}{M}}\varphi_{N,\frac{1}{M}}\left(t+\frac{\frac{1}{M}}{N+\frac{1}{M}}\right)-\frac{\pi\sqrt{-1}}{12 }.
\end{align}
for $0<t<1$, $0<t-s<1$ and $1<t+s<2$. Similarly, one can write the corresponding  expression for the function $V_{N,\frac{1}{M}}$ for the case $0<t<1$ and $0<t\pm s<1$, but which will not used in the following. So we omit it here.  
\end{proposition}

We define the potential function for the twist knot
$\mathcal{K}_p$ as follows
\begin{align}  \label{formula-Vpts}
&V(p,t,s)=\lim_{N\rightarrow\infty}V_{N,\frac{1}{M}}(p,t,s)=\pi \sqrt{-1}\left((2p+1)s^2-(2p+3)s-2t\right)\\\nonumber
    &+\frac{1}{2\pi\sqrt{-1}}\left(\text{Li}_2(e^{2\pi\sqrt{-1}(t+s)})+\text{Li}_2(e^{2\pi\sqrt{-1}(t-s)})-3\text{Li}_2(e^{2\pi\sqrt{-1}t})+\frac{\pi^2}{6}\right). 
\end{align}

\section{Poisson summation formula} \label{Section-Poisson}
In this section, with the help of Poisson summation formula, we write the formula (\ref{formula-coloredJonespotential2}) as a sum of integrals.  First, according to formulas (\ref{formula-coloredJonestwist}) and (\ref{formula-coloredJonespotential2}), we have
\begin{align}
g_{N,\frac{1}{M}}(k,l)&=(-1)^lq^{\frac{k(k+3)}{4}+pl(l+1)}\frac{\{2l+1\}}{\{N\}}\frac{\{k\}!\{N+k\}!}{\{k+l+1\}!\{k-l\}!\{N-k-1\}!}|_{q=\xi_{N,M}}. 
\end{align}

By Lemmas \ref{lemma-varphixi}, \ref{lemma-varphixi2}, \ref{lemma-varphixi3} and formula (\ref{formula-Lambda(t)}), we obtain 
\begin{align}
    \log \left|\{n\}!\right|=-(N+\frac{1}{M})\Lambda\left(\frac{n+\frac{1}{2}}{N+\frac{1}{M}}\right)+O(\log (N+\frac{1}{M}))
\end{align}  
for any integer $0<n\leq N$ and at $q=\xi_{N,\frac{1}{M}}$.

We put
\begin{align}
    v_{N,\frac{1}{M}}(t,s)&=\Lambda\left(t+s-1+\frac{\frac{1}{2}}{N+\frac{1}{M}}\right)+\Lambda\left(t-s+\frac{\frac{1}{2}}{N+\frac{1}{M}}\right)\\\nonumber
    &-\Lambda\left(t-\frac{\frac{1}{M}}{N+\frac{1}{M}}\right)-\Lambda\left(t\right)-\Lambda\left(t+\frac{\frac{1}{M}}{N+\frac{1}{M}}\right), 
\end{align}
then we obtain 
\begin{align}
    |g_{N,\frac{1}{M}}(k,l)|=e^{(N+\frac{1}{M})v_{N,\frac{1}{M}}\left(\frac{k+\frac{1}{2}}{N+\frac{1}{M}},\frac{l+\frac{1}{2}}{N+\frac{1}{M}}\right)+O(\log (N+\frac{1}{M}))}. 
\end{align}

We define the function
\begin{align}
    v(t,s)=\Lambda(t+s)+\Lambda(t-s)-3\Lambda\left(t\right).
\end{align}
Note that $\left(\frac{k+\frac{1}{2}}{N+\frac{1}{M}},\frac{l+\frac{1}{2}}{N+\frac{1}{M}}\right)\in D=\{(t,s)\in \mathbb{R}^2| 1< t+s< 2, 0< t-s<1, \frac{1}{2}< t<1\}$ for $0\leq k,l\leq N-1$. So we may assume 
the function $v(t,s)$ is defined on the region $D$.
We set
    \begin{align}
    D'_0=\{0.02 \leq t-s\leq 0.7, 1.02 \leq t+s\leq 1.7, 0.2 \leq s\leq 0.8,0.5\leq t\leq 0.909\}.
\end{align}

Let $\zeta_{\mathbb{R}}(p)$ be the real part of the critical value $V(p,t_0,s_0)$, see formula (\ref{formula-zetaR(p)}) for its precise definition. 

Then we have
\begin{lemma}[\cite{CZ23-1}, Lemma 4.1] \label{lemma-regionD'0}
The following domain
    \begin{align} \label{formula-domain}
        \left\{(t,s)\in D| v(t,s)> \frac{3.509}{2\pi }\right\}
    \end{align}
    is included in the region $D'_0$.
\end{lemma}

\begin{remark}
We can take $\varepsilon>0$ small enough (such as $\varepsilon=0.00001$), and set 
\begin{align}
    D'_\varepsilon=\left\{0.02+\varepsilon \leq t-s\leq 0.7-\varepsilon, 1.02+\varepsilon \leq t+s\leq 1.7-\varepsilon,\right.\\\nonumber \left. 0.2+\varepsilon \leq s\leq 0.8-\varepsilon,0.5+\varepsilon\leq t\leq 0.909-\varepsilon\right\},
\end{align}
then the region (\ref{formula-domain}) can also be included in the region $D'_{\varepsilon}$. 
\end{remark}

\begin{proposition}[\cite{CZ23-1}, Proposition 4.3] \label{prop-gkl}
For $p\geq 6$ and  $(\frac{k+\frac{1}{2}}{N+\frac{1}{M}},\frac{l+\frac{1}{2}}{N+\frac{1}{M}})\in D\setminus D'_0$,  we have
\begin{align}
    |g_{N,\frac{1}{M}}(k,l)|<O\left(e^{(N+\frac{1}{M})\left(\zeta_{\mathbb{R}}(p)-\epsilon\right)}\right)
\end{align}
for some sufficiently small $\epsilon>0$.
\end{proposition}

For a sufficiently small $\varepsilon$, we take a smooth bump function $\psi$ on $\mathbb{R}^2$ such that
$\psi(t,s)=1$ on $(t,s)\in D'_{\varepsilon}$,  $0<\psi(t,s)<1$ on $(t,s)\in D'_0\setminus D'_{\varepsilon}$, $\psi(t,s)=0$ for $(t,s)\notin D'_0$.  
Let 
\begin{align}
    h_{N,\frac{1}{M}}(k,l)=\psi\left(\frac{k+\frac{1}{2}}{N+\frac{1}{M}},\frac{l+\frac{1}{2}}{N+\frac{1}{M}}\right)g_{N,\frac{1}{M}}(k,l).
\end{align}
Then  by Proposition \ref{prop-gkl},  for $p\geq 6$, we have 
\begin{align}
J_N(\mathcal{K}_p;\xi_{N,\frac{1}{M}})=\sum_{(k,l)\in \mathbb{Z}^2}h_{N,\frac{1}{M}}(k,l)+O\left(e^{(N+\frac{1}{M})\left(\zeta_{\mathbb{R}}(p)-\epsilon\right)}\right).
\end{align}

Note that $h_{N,\frac{1}{M}}$ is $C^{\infty}$-smooth and equals zero outside $D'_0$, it is in the Schwartz space on $\mathbb{R}^2$. By using  Poisson summation formula, we obtain
\begin{proposition} \label{prop-fouriercoeff}
For $p\geq 6$ and $M\geq 2$, the normalized $N$-th colored Jones polynomial of the twist knot $\mathcal{K}_p$ at the root of unity $\xi_{N,\frac{1}{M}}$ is given by 
\begin{align}
J_N(\mathcal{K}_p;\xi_{N,\frac{1}{M}})=\sum_{(m,n)\in \mathbb{Z}^2}\hat{h}_{N,\frac{1}{M}}(m,n)+O\left(e^{(N+\frac{1}{M})\left(\zeta_{\mathbb{R}}(p)-\epsilon\right)}\right)
\end{align}
where
  \begin{align} \label{formula-hathMN}
      \hat{h}_{N,\frac{1}{M}}(m,n)&=(-1)^{m+n+p}e^{\pi\sqrt{-1}(\frac{1}{M}-\frac{1}{4})}\frac{(N+\frac{1}{M})^{\frac{3}{2}}}{\sin\frac{\frac{\pi}{M}}{N+\frac{1}{M}}}\\\nonumber
    &\cdot \int_{D'_0}\psi(t,s)\sin(2\pi s)e^{(N+\frac{1}{M})V_{N,\frac{1}{M}}\left(p,t,s;m,n\right)}dtds
\end{align}
with
\begin{align}
V_{N,\frac{1}{M}}\left(p,t,s;m,n\right)=V_{N,\frac{1}{M}}\left(p,t,s\right)-2\pi\sqrt{-1}mt-2\pi\sqrt{-1}ns,
\end{align}
and 
$V_{N,\frac{1}{M}}(p,t,s)$ is given by formula (\ref{formula-potentialVNM}).

\end{proposition}

We define the function
\begin{align}
    &V(p,t,s;m,n)\\\nonumber 
    &=\lim_{N\rightarrow \infty}V_{N,\frac{1}{M}}(p,t,s;m,n)\\\nonumber
    &=\pi \sqrt{-1}\left((2p+1)s^2-(2p+3+2n)s-(2+2m)t\right)\\\nonumber
    &+\frac{1}{2\pi\sqrt{-1}}\left(\text{Li}_2(e^{2\pi\sqrt{-1}(t+s)})+\text{Li}_2(e^{2\pi\sqrt{-1}(t-s)})-3\text{Li}_2(e^{2\pi\sqrt{-1}t})+\frac{\pi^2}{6}\right). 
\end{align}

\begin{lemma}
We have the following identity
    \begin{align} \label{formula-potientalsym}
          V_{N,\frac{1}{M}}(p,t,1-s;m,n)&=V_{N,\frac{1}{M}}(p,t,s;m,n)-2(n+1)\pi\sqrt{-1}(1-2s)\\\nonumber
          &=V_{N,\frac{1}{M}}(p,t,s;m,-n-2)-2\pi \sqrt{-1}(n+1). 
    \end{align}
\end{lemma}
\begin{proof}
    By a straightforward computation, we obtain the following identity
    \begin{align}
        &\pi\sqrt{-1}\left((2p+1)(1-s)^2-(2p+2n+3)(1-s)+\left(\frac{2}{N+\frac{1}{M}}-2m-2\right)t-\frac{1}{12}\right)\\\nonumber
        &=\pi\sqrt{-1}\left((2p+1)s^2-(2p+2n+3)s+\left(\frac{2}{N+\frac{1}{M}}-2m-2\right)t-\frac{1}{12}\right)\\\nonumber
        &-2(n+1)\pi\sqrt{-1}(1-2s)\\\nonumber
        &=\pi\sqrt{-1}\left((2p+1)s^2-(2p+2(-n-2)+3)s)+\left(\frac{2}{N+\frac{1}{M}}-2m-2\right)t-\frac{1}{12}\right)\\\nonumber
        &-2\pi \sqrt{-1}(n+1).
    \end{align}
    which immediately gives the formula (\ref{formula-potientalsym}).
\end{proof}

Similar to the proof of Proposition 4.6 in \cite{CZ23-1}, we obtain 
\begin{proposition} \label{prop-hathmn}
    For any $m,n\in \mathbb{Z}$, we have
    \begin{align} \label{formula-hnmsy}
        \hat{h}_{N,\frac{1}{M}}(m,-n-2)=-e^{\frac{2\pi\sqrt{-1}(n+1)}{M}}\hat{h}_{N,\frac{1}{M}}(m,n).
    \end{align}
\end{proposition}
\begin{remark}
Formula (\ref{formula-hnmsy}) implies that 
\begin{align}
    \hat{h}_{N,\frac{1}{M}}(m,-1)=0.
\end{align}
This
is the big cancellation. The first situation of such phenomenon of ``Big cancellation" happened in quantum invariants is discovered in the Volume Conjecture of the Turaev-Viro invariants by Chen-Yang \cite{CY18}. 
The hidden reason behind that was found and described as a precise statement of symmetric property of asymptotics of quantum 6j-symbol which is on the Poisson Summation level by Chen-Murakami which is Conjecture 3 in \cite{CJ17}. A special  case of Conjecture 3 in  \cite{CJ17} was proved by Detcherry-Kalfagianni-Yang in \cite{DKY18}.
To the best of our knowledge, this is the first time that such a phenomenon of big cancellation on the Poisson Summation level on the case of colored Jones polynomial is proved.
\end{remark}

\section{Asymptotic expansion at the root of unity $e^{\frac{2\pi\sqrt{-1}}{N+\frac{1}{M}}}$} \label{Section-Asym1}

The goal of this section is to estimate each Fourier coefficients $\hat{h}_N(m,n)$ appearing in Proposition \ref{prop-fouriercoeff}. 

Recall that 
\begin{align}
    &V_{N,\frac{1}{M}}(p,t,s)\\\nonumber
    &=\pi \sqrt{-1}\left((2p+1)s^2-(2p+3)s+\left(\frac{2}{N+\frac{1}{M}}-2\right)t-\frac{6p+4+\frac{12}{M^2}}{12(N+\frac{1}{M})^2}\right)\\\nonumber
    &+\frac{1}{N+\frac{1}{M}}\varphi_{N,\frac{1}{M}}\left(t+s+\frac{\frac{1}{2}}{N+\frac{1}{M}}-1\right)+\frac{1}{N+\frac{1}{M}}\varphi_{N,\frac{1}{M}}\left(t-s+\frac{\frac{1}{2}}{N+\frac{1}{M}}\right)\\\nonumber
    &-\frac{1}{N+\frac{1}{M}}\varphi_{N,\frac{1}{M}}\left(t\right)-\frac{1}{N+\frac{1}{M}}\varphi_{N,\frac{1}{M}}\left(t-\frac{\frac{1}{M}}{N+\frac{1}{M}}\right)\\\nonumber
    &-\frac{1}{N+\frac{1}{M}}\varphi_{N,\frac{1}{M}}\left(t+\frac{\frac{1}{M}}{N+\frac{1}{M}}\right)-\frac{\pi\sqrt{-1}}{12 }.
\end{align}
and 
\begin{align}
    &V(p,t,s;m,n)\\\nonumber 
    &=\pi \sqrt{-1}\left((2p+1)s^2-(2p+3+2n)s-(2+2m)t\right)\\\nonumber
    &+\frac{1}{2\pi\sqrt{-1}}\left(\text{Li}_2(e^{2\pi\sqrt{-1}(t+s)})+\text{Li}_2(e^{2\pi\sqrt{-1}(t-s)})-3\text{Li}_2(e^{2\pi\sqrt{-1}t})+\frac{\pi^2}{6}\right). 
\end{align}
We have 
\begin{lemma} \label{lemma-VMNV}
    For any $L>0$, in the region 
    \begin{align}
        \{(t,s)\in \mathbb{C}^2|(Re(t),Re(s))\in D'_0, |Im t|<L, |Im s|<L\},
    \end{align}
we have 
\begin{align}
     V_{N,\frac{1}{M}}(p,t,s;m,n)&=V(p,t,s;m,n)-\frac{1}{2(N+\frac{1}{M})}\left(\log(1-e^{2\pi\sqrt{-1}(t+s)})\right.\\\nonumber
   &\left.+\log(1-e^{2\pi\sqrt{-1}(t-s)})-4\pi\sqrt{-1}t\right)+\frac{w_{N,\frac{1}{M}}(t,s)}{(N+\frac{1}{M})^2}
\end{align}
    with $|w_{N,\frac{1}{M}}(t,s)|$ bounded from above by a constant independent of $N,M$. 
\end{lemma}
\begin{proof}
By using Taylor expansion, together with Lemma \ref{lemma-varphixi3}, we have 
\begin{align}
    &\varphi_{N,\frac{1}{M}}\left(t+s-1+\frac{\frac{1}{2}}{N+\frac{1}{M}}\right)\\\nonumber
    &=\varphi_{N,\frac{1}{M}}(t+s-1)+\varphi'_{N,\frac{1}{M}}(t+s-1)\frac{\frac{1}{2}}{N+\frac{1}{M}}\\\nonumber&+\frac{\varphi''_{N,\frac{1}{M}}(t+s-1)}{2}\left(\frac{\frac{1}{2}}{N+\frac{1}{M}}\right)^2+O\left(\left(\frac{1}{N+\frac{1}{M}}\right)^2\right)\\\nonumber
    &=\frac{N+\frac{1}{M}}{2\pi\sqrt{-1}}\text{Li}_2(e^{2\pi\sqrt{-1}(t+s)})-\frac{1}{2}\log(1-e^{2\pi\sqrt{-1}(t+s)})\\\nonumber
    &+\frac{\pi\sqrt{-1}}{6(N+\frac{1}{M})}\frac{e^{2\pi\sqrt{-1}(t+s)}}{1-e^{2\pi\sqrt{-1}(t+s)}}+O\left(\left(\frac{1}{N+\frac{1}{M}}\right)^2\right)
    \end{align}    

Similarly, 
we have 
\begin{align}
    &\varphi_{N,\frac{1}{M}}\left(t-s+\frac{\frac{1}{2}}{N+\frac{1}{M}}\right)\\\nonumber
    &=\varphi_{N,\frac{1}{M}}(t-s)+\varphi'_{N,\frac{1}{M}}(t-s)\frac{\frac{1}{2}}{N+\frac{1}{M}}\\\nonumber&+\frac{\varphi''_{N,\frac{1}{M}}(t-s)}{2}\left(\frac{\frac{1}{2}}{N+\frac{1}{M}}\right)^2+O\left(\left(\frac{1}{N+\frac{1}{M}}\right)^2\right)\\\nonumber
    &=\frac{N+\frac{1}{M}}{2\pi\sqrt{-1}}\text{Li}_2(e^{2\pi\sqrt{-1}(t-s)})-\frac{1}{2}\log(1-e^{2\pi\sqrt{-1}(t-s)})\\\nonumber
    &+\frac{\pi\sqrt{-1}}{6(N+\frac{1}{M})}\frac{e^{2\pi\sqrt{-1}(t-s)}}{1-e^{2\pi\sqrt{-1}(t-s)}}+O\left(\left(\frac{1}{N+\frac{1}{M}}\right)^2\right),
    \end{align}    

\begin{align}
    &\varphi_{N,\frac{1}{M}}\left(t+\frac{\frac{1}{M}}{N+\frac{1}{M}}\right)\\\nonumber
    &=\varphi_{N,\frac{1}{M}}(t)+\varphi'_{N,\frac{1}{M}}(t)\frac{\frac{1}{M}}{N+\frac{1}{M}}+\frac{\varphi''_{N,\frac{1}{M}}(t)}{2}\left(\frac{\frac{1}{M}}{N+\frac{1}{M}}\right)^2+O\left(\left(\frac{1}{N+\frac{1}{M}}\right)^2\right)\\\nonumber
    &=\frac{N+\frac{1}{M}}{2\pi\sqrt{-1}}\text{Li}_2(e^{2\pi\sqrt{-1}t})-\frac{1}{M}\log(1-e^{2\pi\sqrt{-1}t})\\\nonumber
    &+\left(\frac{1}{M^2}-\frac{1}{12}\right)\frac{\pi\sqrt{-1}}{(N+\frac{1}{M})}\frac{e^{2\pi\sqrt{-1}t}}{1-e^{2\pi\sqrt{-1}t}}+O\left(\left(\frac{1}{N+\frac{1}{M}}\right)^2\right),
    \end{align}   

\begin{align}
    &\varphi_{N,\frac{1}{M}}\left(t\right)\\\nonumber
    &=\frac{N+\frac{1}{M}}{2\pi\sqrt{-1}}\text{Li}_2(e^{2\pi\sqrt{-1}t})-\frac{\pi\sqrt{-1}}{12(N+\frac{1}{M})}\frac{e^{2\pi\sqrt{-1}t}}{1-e^{2\pi\sqrt{-1}t}}+O\left(\left(\frac{1}{N+\frac{1}{M}}\right)^3\right),
    \end{align}   
    and 
    \begin{align}
    &\varphi_{N,\frac{1}{M}}\left(t-\frac{\frac{1}{M}}{N+\frac{1}{M}}\right)\\\nonumber
    &=\varphi_{N,M}(t)-\varphi'_{N,M}(t)\frac{\frac{1}{M}}{N+\frac{1}{M}}+\frac{\varphi''_{N,M}(t)}{2}\left(\frac{\frac{1}{M}}{N+\frac{1}{M}}\right)^2+O\left(\left(\frac{1}{N+\frac{1}{M}}\right)^2\right)\\\nonumber
    &=\frac{N+\frac{1}{M}}{2\pi\sqrt{-1}}\text{Li}_2(e^{2\pi\sqrt{-1}t})+\frac{1}{M}\log(1-e^{2\pi\sqrt{-1}t})\\\nonumber
    &+\left(\frac{1}{M^2}-\frac{1}{12}\right)\frac{\pi\sqrt{-1}}{(N+\frac{1}{M})}\frac{e^{2\pi\sqrt{-1}t}}{1-e^{2\pi\sqrt{-1}t}}+O\left(\left(\frac{1}{N+\frac{1}{M}}\right)^2\right),
    \end{align}

Therefore, we obtain 
\begin{align}
   &V_{N,\frac{1}{M}}(p,t,s;m,n)\\\nonumber
   &=V(p,t,s;m,n)-\frac{1}{2(N+\frac{1}{M})}\left(\log(1-e^{2\pi\sqrt{-1}(t+s)})\right.\\\nonumber
   &\left.+\log(1-e^{2\pi\sqrt{-1}(t-s)})-4\pi\sqrt{-1}t\right)\\\nonumber
   &-\frac{\pi\sqrt{-1}}{12(N+\frac{1}{M})^2}\left(-2\frac{e^{2\pi\sqrt{-1}(t+s)}}{1-e^{2\pi\sqrt{-1}(t+s)}}-2\frac{e^{2\pi\sqrt{-1}(t-s)}}{1-e^{2\pi\sqrt{-1}(t-s)}}+\left(\frac{24}{M^2}-3\right)\frac{e^{2\pi\sqrt{-1}t}}{1-e^{2\pi\sqrt{-1}t}}\right.\\\nonumber &\left.+6p+4+\frac{12}{M^2}\right)+O\left(\frac{1}{(N+\frac{1}{M})^3}\right), 
\end{align}

Finally, we let 
\begin{align}
    &w_{N,\frac{1}{M}}(t,s)\\\nonumber
    &=-\frac{\pi\sqrt{-1}}{12}\left(-2\frac{e^{2\pi\sqrt{-1}(t+s)}}{1-e^{2\pi\sqrt{-1}(t+s)}}-2\frac{e^{2\pi\sqrt{-1}(t-s)}}{1-e^{2\pi\sqrt{-1}(t-s)}}+\left(\frac{24}{M^2}-3\right)\frac{e^{2\pi\sqrt{-1}t}}{1-e^{2\pi\sqrt{-1}t}}\right.\\\nonumber &\left.+6p+4+\frac{12}{M^2}\right)+O\left(\frac{1}{(N+\frac{1}{M})}\right),
\end{align}
and we finish the proof Lemma \ref{lemma-VMNV}. 
\end{proof}

We consider the critical point of $V(p,t,s)$, which is a solution of the following equations

\begin{align}  \label{equation-critical1}
    \frac{\partial V(p,t,s)}{\partial t}&=-2\pi\sqrt{-1}+3\log(1-e^{2\pi\sqrt{-1}t})\\\nonumber
    &-\log(1-e^{2\pi\sqrt{-1}(t+s)})-\log(1-e^{2\pi\sqrt{-1}(t-s)})=0,
\end{align}
\begin{align} \label{equation-critical2}
    \frac{\partial V(p,t,s)}{\partial s}&=(4p+2)\pi\sqrt{-1}s-(2p+3)\pi\sqrt{-1}\\\nonumber
    &-\log(1-e^{2\pi\sqrt{-1}(t+s)})+\log(1-e^{2\pi\sqrt{-1}(t-s)})=0. 
\end{align}

\begin{proposition}[\cite{CZ23-1}, Proposition 5.3]  \label{prop-critical}
 The critical point equations (\ref{equation-critical1}), (\ref{equation-critical2}) has a unique solution $(t_0,s_0)=(t_{0R}+X_0\sqrt{-1},s_{0R}+Y_0\sqrt{-1})$ with $(t_{0R},s_{0R})$ lies in the region $D'_0$.     
\end{proposition}

Now we set 
$\zeta(p)$ to be the critical value of the potential function $V(p,t,s)$, i.e. 
\begin{align}
    \zeta(p)=V(p,t_0,s_0),
\end{align}
and set
\begin{align} \label{formula-zetaR(p)}
    \zeta_\mathbb{R}(p)=Re \zeta(p)=Re V(p,t_0,s_0). 
\end{align}

Note that$V_{N,\frac{1}{M}}(p,t,s;m,n)$ converges to $V(p,t,s;m,n)$ uniformly. By Lemma \ref{lemma-VMNV} and Remark \ref{remark-saddle}, we only need to verify the assumption of Proposition \ref{proposition-saddlemethod} for the function $V(p,t,s;m,n)$   which has been done in \cite{CZ23-1}. Hence, as in \cite{CZ23-1}, one can also obtain 
\begin{theorem} \label{theorem-main1}
For $p\geq 6$ and $M\geq 2$, the asymptotic expansion of the colored Jones polynomial of the twist knot $\mathcal{K}_p$ at the root of unity $\xi_{N,\frac{1}{M}}$ is given by 
    \begin{align}
        J_{N}(\mathcal{K}_p;\xi_{N,\frac{1}{M}})&=(-1)^{p}e^{\pi\sqrt{-1}(\frac{1}{4}+\frac{2}{M})}\frac{4\pi \left(N+\frac{1}{M}\right)^{\frac{1}{2}}\sin\frac{\pi}{M}}{\sin\frac{\frac{\pi}{M}}{N+\frac{1}{M}}}\omega(p)e^{(N+\frac{1}{M})\zeta(p)}\\\nonumber
       &\cdot\left(1+\sum_{i=1}^d\kappa_i(p,\frac{1}{M})\left(\frac{2\pi\sqrt{-1}}{N+\frac{1}{M}}\right)^i+O\left(\frac{1}{(N+\frac{1}{M})^{d+1}}\right)\right),
    \end{align}
    for $d\geq 1$, where $\omega(p)$ and $\kappa_i(p,\frac{1}{M})$ are constants determined by $\mathcal{K}_p$.
For example
\begin{align} \label{formula-omega}
    \omega(p)&=\frac{\sin (2\pi s_0)e^{2\pi\sqrt{-1}t_0}}{(1-e^{2\pi\sqrt{-1}t_0})^{\frac{3}{2}}\sqrt{\det Hess(V)(t_0,s_0)}}\\\nonumber
    &=\frac{(y_0-y_0^{-1})x_0}{-4\pi (1-x_0)^\frac{3}{2}\sqrt{H(p,x_0,y_0)}}
\end{align}
with
\begin{align}
    H(p,x_0,y_0)&=\left(\frac{-3(2p+1)}{\frac{1}{x_0}-1}+\frac{2p+1}{\frac{1}{x_0y_0}-1}+\frac{2p+1}{\frac{1}{x_0/y_0}-1}-\frac{3}{(\frac{1}{x_0}-1)(\frac{1}{x_0y_0}-1)}\right.\\\nonumber
    &\left.-\frac{3}{(\frac{1}{x_0}-1)(\frac{1}{x_0/y_0}-1)}+\frac{4}{(\frac{1}{x_0y_0}-1)(\frac{1}{x_0/y_0}-1)}\right),
\end{align}
by letting $x_0=e^{2\pi\sqrt{-1}t_0}$ and $y_0=e^{2\pi\sqrt{-1}s_0}$.
\end{theorem}

\section{Asymptotic expansion at the root of unity $e^{\frac{2\pi\sqrt{-1}}{N}}$} \label{Section-Asym2}

According to Proposition \ref{prop-fouriercoeff}, we  can write the colored Jones polynomial $J_{N}(\mathcal{K}_p;\xi_{N,\frac{1}{M}})$ of the twist knot $\mathcal{K}_p$  at the root of unity $\xi_{N,\frac{1}{M}}$ as the summation of two parts
\begin{align} \label{formula-JNtwo}
    J_{N}(\mathcal{K}_p;\xi_{N,\frac{1}{M}})&=\sum_{(m,n)\in \mathbb{Z}^2}\hat{h}_{N,\frac{1}{M}}(m,n)+R_{N,\frac{1}{M}},
\end{align}
where 
\begin{align}
 R_{N,\frac{1}{M}}&=\sum_{\substack{(k,l)\in \mathbb{Z}^2\\ (\frac{k+\frac{1}{2}}{N+\frac{1}{M}},\frac{l+\frac{1}{2}}{N+\frac{1}{M}})\in D\setminus D'_0}}g_{N,\frac{1}{M}}(k,l).
 \end{align}

\begin{lemma} \label{lemma-R}
There exists a constant $C_1$ independent of $M$, such that 
    \begin{align}
     |R_{N,\frac{1}{M}}| \leq C_1e^{N\left(\zeta_{\mathbb{R}}(p)-\epsilon\right)}.    
    \end{align}
for some sufficiently small $\epsilon$.  
\end{lemma}
\begin{proof}
For any $(k,l)\in \mathbb{Z}^2$ with $(\frac{k+\frac{1}{2}}{N+\frac{1}{M}},\frac{l+\frac{1}{2}}{N+\frac{1}{M}})\in D\setminus D'_0$,  by modifying the proof of  Proposition 4.3 in \cite{CZ23-1} slightly, we obtain
\begin{align}
    |g_{N,\frac{1}{M}}(k,l)|<Ce^{(N+\frac{1}{M})(\zeta_{\mathbb{R}}(p)-\epsilon')},
\end{align}
where $C$ is a constant independent of $N$ and $M$.
Therefore, we obtain 
\begin{align}
    |R_{N,\frac{1}{M}}|&\leq \sum_{\substack{(k,l)\in \mathbb{Z}^2\\ (\frac{k+\frac{1}{2}}{N+\frac{1}{M}},\frac{l+\frac{1}{2}}{N+\frac{1}{M}})\in D\setminus D'_0}}|g_{N,\frac{1}{M}}(k,l)|\leq (N+\frac{1}{M})^2Ce^{(N+\frac{1}{M})(\zeta_{\mathbb{R}}(p)-\epsilon')}\\\nonumber
    &\leq (N+\frac{1}{2})^2Ce^{(N+\frac{1}{2})(\zeta_{\mathbb{R}}(p)-\epsilon')}\leq C_1e^{N(\zeta_{\mathbb{R}}(p)-\epsilon)}
\end{align}
for some constant $C_1$ independent of $M$ ( depends on $N$), where we have let $\epsilon=\frac{\epsilon'}{2}$.
\end{proof}

\begin{lemma}\label{lemma-Ssum}
There exists a constant $C_2$ independent of $M$, such that 
    \begin{align}
     |\sum_{(m,n)\in \mathbb{Z}^2}\hat{h}_{N,\frac{1}{M}}(m,n)| \leq C_2.    
    \end{align}
\end{lemma} 
\begin{proof}
    We introduce the set 
    \begin{align}
        \mathcal{S}=\{(m,n)\in \mathbb{Z}^2|n\geq 0 \}. 
    \end{align}
By using Proposition \ref{prop-hathmn}, we obtain
\begin{align}
   \sum_{(m,n)\in \mathbb{Z}^2}\hat{h}_{N,\frac{1}{M}}(m,n)=\sum_{(m,n)\in \mathcal{S}}(1-e^{\frac{2\pi\sqrt{-1}(n+1)}{M}})\hat{h}_{N,\frac{1}{M}}(m,n).
\end{align}

 Therefore, we only need  to prove that
\begin{align}
    \sum_{(m,n)\in \mathcal{S}}|(1-e^{\frac{2\pi\sqrt{-1}(n+1)}{M}})\hat{h}_{N,\frac{1}{M}}(m,n)|\leq C_2
\end{align}
    for some constant $C_2$ independent of  $M$. 

Note that, by formula (\ref{formula-hathMN}), we have
\begin{align}
    &|(1-e^{\frac{2\pi\sqrt{-1}(n+1)}{M}})\hat{h}_{N,\frac{1}{M}}(m,n)|\\\nonumber
    &\leq \frac{2\sin \frac{(n+1)\pi}{M}}{\sin \frac{\frac{\pi}{M}}{N+\frac{1}{M}}}(N+\frac{1}{M})^{\frac{3}{2}}\left|\int_{D'_0}\psi(t,s)\sin(2\pi s)e^{(N+\frac{1}{M})V_{N,\frac{1}{M}}\left(p,t,s;m,n\right)}dtds\right|
\end{align}

By using the inequalities 
\begin{align}
    \sin(x)<x,\ \sin (x)>x-\frac{x^3}{3!}
\end{align}
for $x>0$, and $M\geq 2$, we obtain 
\begin{align}
    \frac{\sin \frac{(n+1)\pi}{M}}{\sin \frac{\frac{\pi}{M}}{N+\frac{1}{M}}}&\leq \frac{\frac{(n+1)\pi}{M}}{\left(\frac{\frac{\pi}{M}}{N+\frac{1}{M}}-\frac{1}{6}\frac{(\frac{\pi}{M})^3}{(N+\frac{1}{M})^3}\right)}=\frac{n+1}{\frac{1}{N+\frac{1}{M}}-\frac{1}{6}\frac{(\frac{\pi}{M})^2}{(N+\frac{1}{M})^3}}\\\nonumber
    &\leq \frac{n+1}{\frac{1}{N+\frac{1}{2}}-\frac{1}{6}\frac{(\frac{\pi}{2})^2}{(N+\frac{1}{2})^3}}= \frac{n+1}{\frac{1}{N+\frac{1}{2}}\left(1-\frac{\pi^2}{24}\frac{1}{(N+\frac{1}{2})^2}\right)}\\\nonumber
    &\leq \frac{n+1}{\frac{1}{N+\frac{1}{2}}\left(1-\frac{\pi^2}{24}\frac{1}{(\frac{3}{2})^2}\right)}=\frac{(n+1)(N+\frac{1}{2})}{\frac{54-\pi^2}{54}}.
\end{align}

Therefore, we have
\begin{align}
    &|(1-e^{\frac{2\pi\sqrt{-1}(n+1)}{M}})\hat{h}_{N,\frac{1}{M}}(m,n)|\\\nonumber
    &\leq \frac{2\sin \frac{(n+1)\pi}{M}}{\sin \frac{\frac{\pi}{M}}{N+\frac{1}{M}}}(N+\frac{1}{M})^{\frac{3}{2}}\left|\int_{D'_0}\psi(t,s)\sin(2\pi s)e^{(N+\frac{1}{M})V_{N,\frac{1}{M}}\left(p,t,s;m,n\right)}dtds\right|\\\nonumber
    &\leq \frac{108}{54-\pi^2} (n+1)(N+\frac{1}{2})^{\frac{5}{2}}\left|\int_{D'_0}\psi(t,s)\sin(2\pi s)e^{(N+\frac{1}{M})V_{N,\frac{1}{M}}\left(p,t,s;m,n\right)}dtds\right|.
\end{align}

Let $\Delta=\frac{\partial^2}{\partial t^2}+\frac{\partial^2}{\partial s^2 }$ be the Laplacian operator. Let $l \geq 1$ be an integer, $\psi(t,s)$ is a bump function which vanishes on the boundary of $D'_{0}$,  we have
\begin{align} \label{formula-Laplacian}
&\int_{D'_0}\psi(t,s)\sin(2\pi s)e^{(N+\frac{1}{M})V_{N,\frac{1}{M}}(p,t,s)}\Delta^l\left(e^{-2(N+\frac{1}{M})\left(m\pi \sqrt{-1}t+n\pi \sqrt{-1}s\right)}\right)dtds\\\nonumber
&=\int_{D'_0}\left(\Delta^l\left( \psi(t,s)\sin(2\pi s)e^{(N+\frac{1}{M})V_{N,\frac{1}{M}}(p,t,s)}\right)\right)e^{-2(N+\frac{1}{M})\left(m\pi \sqrt{-1}t+n\pi \sqrt{-1}s\right)}dtds.
\end{align}
In the following, it is enough to use the above formula in the case of $l=2$.

Clearly, the following identity holds
\begin{align}
    &\int_{D'_0}\psi(t,s)\sin(2\pi s)e^{(N+\frac{1}{M})V_{N,\frac{1}{M}}(p,t,s)}\left(\Delta^2\left(e^{-2(N+\frac{1}{M})\left(m\pi \sqrt{-1}t+n\pi \sqrt{-1}s\right)}\right)\right)dtds\\\nonumber
    &=(2\pi(N+\frac{1}{M}))^4(m^2+n^2)^2\int_{D'_0}\psi(t,s)\sin(2\pi s)e^{(N+\frac{1}{M})\left(V_{N,\frac{1}{M}}(p,t,s;m,n)\right)}dtds.
\end{align}

Therefore, for $(m,n)\in \mathcal{S}$, we have
\begin{align}  \label{formula-U}
&\int_{D'_0}\psi(t,s)\sin(2\pi s)e^{(N+\frac{1}{M})V_{N,\frac{1}{M}}\left(p,t,s;m,n\right)}dtds\\\nonumber
    &=\int_{D'_0}\psi(t,s)\sin(2\pi s)e^{(N+\frac{1}{M})\left(V_{N,\frac{1}{M}}(p,t,s)-2\pi m\sqrt{-1}t-2\pi n\sqrt{-1}s\right)}dtds\\\nonumber
    &=\frac{1}{(2\pi (N+\frac{1}{M}))^4(m^2+n^2)^2}\\\nonumber
    &\cdot\int_{D'_0}\psi(t,s)\sin(2\pi s)e^{(N+\frac{1}{M})V_{N,\frac{1}{M}}(p,t,s)}\left(\Delta^2 e^{-2(N+\frac{1}{M})\left(m\pi \sqrt{-1}t+n\pi \sqrt{-1}s\right)}\right)dtds\\\nonumber
    &=\frac{1}{(2\pi (N+\frac{1}{M}))^4(m^2+n^2)^2}\\\nonumber
    &\cdot \int_{D'_0}\left(\Delta^2\left( \psi(t,s)\sin(2\pi s)e^{(N+\frac{1}{M})V_{N,\frac{1}{M}}(p,t,s)}\right)\right)e^{-2(N+\frac{1}{M})\left(m\pi \sqrt{-1}t+n\pi \sqrt{-1}s\right)}dtds\\\nonumber
    &=\frac{1}{(2\pi (N+\frac{1}{M}))^4(m^2+n^2)^2}\int_{D'_0}\tilde{U}_{N,\frac{1}{M}}(p,t,s)e^{(N+\frac{1}{M})\left(V_{N,\frac{1}{M}}(p,t,s;m,n)\right)}dtds,
\end{align}
where 
\begin{align}
    \tilde{U}_{N,\frac{1}{M}}(p,t,s)=e^{-(N+\frac{1}{M})V_{N,\frac{1}{M}}(p,t,s)}\Delta^2\left( \psi(t,s)\sin(2\pi s)e^{(N+\frac{1}{M})V_{N,\frac{1}{M}}(p,t,s)}\right)
\end{align}
which is a smooth function independent of $m$ and $n$.

Note that 
\begin{align}
    \lim_{M\rightarrow \infty}V_{N,\frac{1}{M}}(p,t,s)=V_{N,0}(p,t,s)
\end{align}
and
\begin{align}
 \lim_{M\rightarrow \infty}\tilde{U}_{N,\frac{1}{M}}(p,t,s)=\tilde{U}_{N,0}(p,t,s)=e^{-NV_{N,0}(p,t,s)}\Delta^2\left( \psi(t,s)\sin(2\pi s)e^{NV_{N,0}(p,t,s)}\right).
\end{align}
Let $x=\frac{1}{M}$, then
$V_{N,x}(p,t,s)$ and 
$\tilde{U}_{N,x}(p,t,s)$ can be viewed as a continuous function of $(x,t,s)\in [0,\frac{1}{2}]\times D'_0$.

So we can take
\begin{align}
    C'=\max_{(x,t,s)\in [0,\frac{1}{2}]\times D'_0}\tilde{U}_{N,x}(p,t,s),
\end{align}
and
\begin{align}
    C''=\max_{(x,t,s)\in [0,\frac{1}{2}]\times D'_0}|ReV_{N,x}(p,t,s)|,
\end{align}
then $C'$  and $C''$ are two constants independent of $M$. 
    
    Finally, we obtain 
     \begin{align} \label{formula-final1}
         &\sum_{(m,n)\in \mathcal{S}}\left|(1-e^{\frac{2\pi\sqrt{-1}(n+1)}{M}})\hat{h}_{N,\frac{1}{M}}(m,n)\right|\\\nonumber
         &\leq \frac{108}{54-\pi^2}(N+\frac{1}{2})^{\frac{5}{2}}\sum_{(m,n)\in \mathcal{S}}(n+1)
        \left|\int_{D'_0}\psi(t,s)\sin(2\pi s)e^{(N+\frac{1}{M})V_{N,\frac{1}{M}}(p,t,s;m,n)}dtds\right|\\\nonumber
        &=\frac{108}{54-\pi^2}\frac{(N+\frac{1}{2})^\frac{5}{2}}{(2\pi (N+\frac{1}{M}))^4}\sum_{(m,n)\in \mathcal{S}}\frac{n+1}{(m^2+n^2)^2}\left|\int_{D'_0}\tilde{U}_{N,\frac{1}{M}}(p,t,s)e^{(N+\frac{1}{M})V_{N,\frac{1}{M}}(p,t,s;m,n)}dtds\right|\\\nonumber
        &\leq \frac{108}{54-\pi^2}\frac{(N+\frac{1}{2})^\frac{5}{2}}{(2\pi N)^4}\sum_{(m,n)\in \mathcal{S}}\frac{n+1}{(m^2+n^2)^2}C'e^{(N+\frac{1}{2})C''}A(D'_0),
    \end{align}
    where $A(D'_0)$ in the last $``\leq "$ denotes the area of the region $D'_0$. 
 Let
    \begin{align}
    C_2=\frac{108}{54-\pi^2}\frac{(N+\frac{1}{2})^\frac{5}{2}}{(2\pi N)^4}\sum_{(m,n)\in \mathcal{S}}\frac{n+1}{(m^2+n^2)^2}C'e^{(N+\frac{1}{2})C''}A(D'_0),
    \end{align}
    
    since the power series $\sum_{(m,n)\in \mathcal{S}}\frac{n+1}{(m^2+n^2)^2}$ is convergent,  $C_2$ is a constant independent of $M$. Hence we prove Lemma \ref{lemma-Ssum}.
\end{proof}
\begin{remark} \label{remark-converges}
    By the above formula (\ref{formula-final1}), if we let 
    \begin{align}
        C_3=C'e^{(N+\frac{1}{2})C''}A(D'_0)\frac{108}{54-\pi^2}\frac{(N+\frac{1}{2})^\frac{5}{2}}{(2\pi N)^4},
    \end{align}
    then $C_3$ is a constant independent of $M$, and we have
    \begin{align}
        \sum_{(m,n)\in \mathcal{S}}\left|(1-e^{\frac{2\pi\sqrt{-1}(n+1)}{M}})\hat{h}_{N,\frac{1}{M}}(m,n)\right|\leq C_3\sum_{(m,n)\in \mathcal{S}}\frac{n+1}{(m^2+n^2)^2}.
    \end{align}
    Hence the series $\sum_{(m,n)\in \mathcal{S}}(1-e^{\frac{2\pi\sqrt{-1}(n+1)}{M}})\hat{h}_{N,\frac{1}{M}}(m,n)$ is uniformly converges with respect to $x=\frac{1}{M}$. 
\end{remark}

\begin{lemma} \label{lemma-Lebes}
The following identity holds
\begin{align}
   &\lim_{M\rightarrow \infty} \int_{D'_0}\psi(t,s)\sin(2\pi s)e^{(N+\frac{1}{M})V_{N,\frac{1}{M}}\left(p,t,s;m,n\right)}dtds\\\nonumber
   &=\int_{D'_0}\psi(t,s)\sin(2\pi s)\lim_{M\rightarrow \infty}e^{(N+\frac{1}{M})V_{N,\frac{1}{M}}\left(p,t,s;m,n\right)}dtds\\\nonumber
   &=\int_{D'_0}\psi(t,s)\sin(2\pi s)e^{NV_{N,0}\left(p,t,s;m,n\right)}dtds.
\end{align}
\end{lemma}
\begin{proof}
    Note that the limit
    \begin{align}
    \lim_{M\rightarrow \infty}\psi(t,s)\sin(2\pi s)e^{(N+\frac{1}{M})V_{N,\frac{1}{M}}\left(p,t,s;m,n\right)}=\psi(t,s)\sin(2\pi s)e^{NV_{N,0}\left(p,t,s;m,n\right)}
    \end{align}
    exists, where 
    \begin{align}  \label{formula-VN0}
    V_{N,0}(p,t,s)&=\lim_{M\rightarrow \infty}V_{N,\frac{1}{M}}(p,t,s)\\\nonumber
    &=\pi \sqrt{-1}\left((2p+1)s^2-(2p+3)s+\left(\frac{3}{N}-2\right)t-\frac{3p+2}{6N^2}\right)\\\nonumber
    &+\frac{1}{N}\varphi_{N,0}\left(t+s+\frac{1}{2N}-1\right)+\frac{1}{N}\varphi_{N,0}\left(t-s+\frac{1}{2N}\right)\\\nonumber
    &-\frac{3}{N}\varphi_{N,0}\left(t\right)-\frac{\pi\sqrt{-1}}{12 }.
\end{align}
Moreover, let $x=\frac{1}{M}$,  $V_{N,x}(p,t,s)$ is a continuous function on $[0,\frac{1}{2}]\times D'_0$. We take $C=\max_{(x,t,s)\in [0,\frac{1}{2}]\times D'_{0}}|Re V_{N,x}(p,t,s;m,n)|$. Then we obtain
\begin{align}
   | \psi(t,s)\sin(2\pi s)e^{(N+\frac{1}{M})V_{N,\frac{1}{M}}\left(p,t,s;m,n\right)}|<e^{(N+\frac{1}{2})C}.
\end{align}
Clearly, the upper bound $e^{(N+\frac{1}{2})C}$ is independent of $M$. Therefore, by Lebesgue dominated convergence theorem, we prove Lemma \ref{lemma-Lebes}. 
\end{proof}

\begin{proposition} \label{prop-JNlimit}
For $p\geq 6$,  the colored Jones polynomial of the twist knot $\mathcal{K}_p$ at the root of unity $\xi_{N,0}$ is given by
    \begin{align}      J_N(\mathcal{K}_p;\xi_{N,0})=\lim_{M\rightarrow \infty}J_{N}(\mathcal{K}_p;\xi_{N,\frac{1}{M}})=\sum_{(m,n)\in \mathcal{S}}\tilde{h}_{N,0}(m,n)+O(e^{N\left(\zeta_{\mathbb{R}}(p)-\epsilon\right)}),  
    \end{align}
where \begin{align}
        \mathcal{S}=\{(m,n)\in \mathbb{Z}^2|n\geq 0 \}
    \end{align}
    and
    \begin{align}  \label{formula-tildehN}
        \tilde{h}_{N,0}(m,n)&=(-1)^{m+n+p+1}2e^{\frac{\pi\sqrt{-1}}{4}}(n+1)N^{\frac{5}{2}} \int_{D'_0}\psi(t,s)\sin(2\pi s)e^{NV_{N,0}\left(p,t,s;m,n\right)}dtds.
\end{align}
\end{proposition}
\begin{proof}
   By formula (\ref{formula-JNtwo}) and Proposition \ref{prop-hathmn},  we compute
\begin{align} \label{formula-JNthree}
    J_{N}(\mathcal{K}_p;\xi_{N,\frac{1}{M}})&=\sum_{(m,n)\in \mathbb{Z}^2}\hat{h}_{N,\frac{1}{M}}(m,n)+R_{N,\frac{1}{M}}\\\nonumber
    &=\sum_{(m,n)\in \mathcal{S}}\tilde{h}_{N,\frac{1}{M}}(m,n)+R_{N,\frac{1}{M}}
\end{align}
where
\begin{align}
    \tilde{h}_{N,\frac{1}{M}}(m,n)&=(1-e^{\frac{2\pi\sqrt{-1}(n+1)}{M}})\hat{h}_{N,\frac{1}{M}}(m,n)\\\nonumber
    &=\sum_{(m,n)\in \mathcal{S}}(-1)^{m+n+p+1}2e^{\frac{(n+2)\pi\sqrt{-1}}{M}}e^{\frac{\pi\sqrt{-1}}{4}}\sin\left(\frac{(n+1)\pi}{M}\right)\frac{(N+\frac{1}{M})^\frac{3}{2}}{\sin \frac{\frac{\pi}{M}}{N+\frac{1}{M}}}\\\nonumber
    &\cdot \int_{D'_0}\psi(t,s)\sin(2\pi s)e^{(N+\frac{1}{M})V_{N,\frac{1}{M}}\left(p,t,s;m,n\right)}dtds.
\end{align}

Therefore, by Lemma \ref{lemma-Ssum} and Remark \ref{remark-converges}, we obtain 
\begin{align}
    J_{N}(\mathcal{K}_p;\xi_{N,0})&=\lim_{M\rightarrow \infty} J_{N}(\mathcal{K}_p;\xi_{N,\frac{1}{M}}) \\\nonumber
    &=\sum_{(m,n)\in \mathcal{S}}\lim_{M\rightarrow \infty}\tilde{h}_{N,\frac{1}{M}}(m,n)+\lim_{M\rightarrow \infty}R_{N,\frac{1}{M}}.
\end{align}
Furthermore, by using Lemma \ref{lemma-Lebes} and Lemma \ref{lemma-R} respectively, we get
\begin{align}
\lim_{M\rightarrow \infty}\tilde{h}_{N,\frac{1}{M}}(m,n)=\tilde{h}_{N,0}(m,n),
    \end{align}
    where $\tilde{h}_{N,0}(m,n)$ is given by formula (\ref{formula-tildehN}), 
and 
\begin{align}
    |\lim_{M\rightarrow \infty}R_{N,\frac{1}{M}}|<C_1e^{N\left(\zeta_{\mathbb{R}}(p)-\epsilon\right)}.
\end{align}

Combining the above formulas together, we obtain Proposition \ref{prop-JNlimit}. 
\end{proof}

Similar to the proof of Lemma \ref{lemma-VMNV}, we see that $V_{N,0}(p,t,s;m,n)$ converges to the potential function $V(p,t,s;m,n)$ uniformly on $D'_0$. We can apply the results  in \cite{CZ23-1} to estimate every integral appearing in the  Fourier coefficients $\tilde{h}_{N,0}(m,n)$ of Proposition \ref{prop-JNlimit},
\begin{align}    \int_{D'_0}\psi(t,s)\sin(2\pi s)e^{NV_{N,0}\left(p,t,s;m,n\right)}dtds.
\end{align}
for $(m,n)\in \mathcal{S}$. 

Finally, we obtain 
\begin{theorem}
For $p\geq 6$, the asymptotic expansion of the colored Jones polynomial of the twist knot $\mathcal{K}_p$ at the root of unity $e^{\frac{2\pi \sqrt{-1}}{N}}$ is given by  the following form
\begin{align}
  J_{N}(\mathcal{K}_p;e^{\frac{2\pi \sqrt{-1}}{N}})&=(-1)^p4\pi e^{\frac{\pi\sqrt{-1}}{4}}N^{\frac{3}{2}}\omega(p)e^{N\zeta(p)}\\\nonumber
       &\cdot\left(1+\sum_{i=1}^d\kappa_i(p)\left(\frac{2\pi\sqrt{-1}}{N}\right)^i+O\left(\frac{1}{N^{d+1}}\right)\right),
    \end{align}
    for $d\geq 1$, where $\omega(p)$ and $\kappa_i(p)$ are constants determined by $\mathcal{K}_p$.
\end{theorem}

\section{Related Questions}
In this final section, we give two
related questions which deserve to be studied further.

1. The first natural question is to study the asymptotic expansion formula for the colored Jones polynomial of the 
 double twist knot $\mathcal{K}_{p,s}$ at the different roots of unity. 
 
2. It is also interesting to study the asymptotic expansion for the Reshetikhin-Turaev invariants of the closed hyperbolic 3-manifolds obtained by $\frac{p}{q}$-surgery along the twist knot at the root of unity   $e^{\frac{4\pi\sqrt{-1}}{r}}$.

3. We prove the volume conjecture for twist knot $\mathcal{K}_p$ with $p\geq 6$ in this paper. Ohtsuki's work \cite{Oht16,Oht17} imply that the volume conjecture for twist knots $\mathcal{K}_2=5_2$ and  $\mathcal{K}_3=7_2$ holds. So the volume conjecture for twist knot $\mathcal{K}_4$ and $\mathcal{K}_5$  is still open. 

4. As to the twist knot $\mathcal{K}_{p}$ with $p\leq -1$, the volume conjecture for $\mathcal{K}_{-1}=4_1$ was proved firstly by Ekholm, $\mathcal{K}_{-2}=6_1$ was proved in \cite{OhtYok18}.     Our method seems can not be applied to the cases $\mathcal{K}_{-3}=8_1$ and $\mathcal{K}_{-4}=10_1$ directly, so from our point of view, it is also difficult to prove the volume conjecture for them.

\end{document}